\documentclass{amsart}
\usepackage{graphicx}
\newtheorem{thm}{Theorem}[section]

\newtheorem{lem}[thm]{Lemma}

\theoremstyle{definition}
\newtheorem{defn}[thm]{Definition}
\theoremstyle{remark}
\newtheorem{rem}[thm]{Remark}
\numberwithin{equation}{section}

\begin{document}

\title[Solovay model and duality principle
]{Solovay model and duality principle between the measure and the Baire
category  in a Polish topological vector space $H(X,S,\mu)$}%
\author{G. Pantsulaia}%
\address{I.~Vekua Institute of Applied
Mathematics, Tbilisi State University, Tbilisi DC 0143, Georgia}%
\email{gogipantsulaia@yahoo.com}%

\thanks{This paper was partially supported by Shota
Rustaveli National Science Foundation's Grant no FR/116/5-100/14}%

\begin{abstract}In Solovay model it is shown that  the
duality principle between the measure and the Baire category
holds true with respect to the sentence - '' {\it  The domain of an arbitrary  generalized  integral for a vector-function is of first category.}''

\end{abstract}
\maketitle
\section{Introduction}

In \cite{Solovay1970}, has been  showed  that the existence of a non-Lebesgue measurable set cannot be proved in Zermelo-Frankel set theory $ZF$
if use of the axiom of choice is disallowed. In fact, even adjoining an axiom $DC$ to  $ZF$,  which  allows countably many consecutive choices,
does not  create a theory  strong enough to construct  a non-measurable  set. S. Shelah \cite{Shelah1984} proved that an inaccessible cardinal is
necessary to build a model of set theory in which every set of reals is Lebesgue measurable. A simpler and metamathematically free proof of
Shelah's result was given in \cite{Jean84}. As a corollary, it was proved  that the existence of  an uncountable well ordered set of reals
implies the existence of a non-measurable set of reals.

Here are several thousand papers devoted to investigations  in Solovay model.
For example, in \cite{Wright73} has been proved that in such a model each linear operator on a Hilbert
space is a bounded linear operator. It was shown in \cite{Pan85-3} that the domain of an arbitrary generalized integral in the same model is the
set of the first category.  In \cite{Pan04-1}, an  example of a non-zero non-atomic translation-invariant Borel measure $\nu_p$  on the Banach
space $l_p(1 \le p <\infty)$ is constructed in Solovay's model such that  the condition ''$\nu_p$-almost every element of $l_p$ has a property
$P$ '' implies that ''{\it almost every element of $l_p$} (in the sense of \cite{HSY92}) has the property $P$''.

In order to formulate the main goal of the present paper, we introduce a new approach which can be considered as a certain modification
of the {\it duality principle between the measure and the Baire category}  introduced in \cite{Oxt70}: Let $(E,\tau)$ be a topological vector space. Denote by
$B(E)$ the Borel $\sigma$-algebra of subsets of the space $E$.
A universally measurable subset of $A$  is called small
in the sense of measure if $A$ is Haar null (equivalently, shy) set. Analogously, a subset $B
\subseteq E$ is called small in the sense of category if it is of
first category set in $E$. Further, let $P$ be a sentence
formulated only by using the notions of measure zero, of the first
category and of purely set-theoretical notions. We say that the
duality principle between the measure and the Baire category
holds true with respect to the sentence $P$ if the sentence $P$ is
equivalent to the sentence $P^*$ obtained from the sentence $P$ by
interchanging in it the notions of the above small sets.

The  present paper is devoted to study of the question asking whether the
duality principle holds true between the measure and the Baire category
with respect to the sentence - '' {\it  The domain of an arbitrary  generalized  integral for a vector-function is  of  first category .}''

The paper is organized as follows.

In Section 2 we consider some auxiliary notions and facts from the measure theory and topology. In Section 3 we give the proof of the main result.

\section{Some auxiliary notions and facts from measure theory and topology}

\begin{defn}Let $(E,\tau)$ be a topological space. We say that a subset $Y \subseteq E $ has Baire property if $Y$ admits the following representation
$$
Y=(Z \cup Z^{'})\setminus Z^{''},
$$
where $Z$ is an open set in $E$, $Z^{'}$ and $ Z^{''}$ are sets of the first category.
\end{defn}

\begin{defn} We say that a topological space $(E,\tau)$ is Baire space if an arbitrary non-empty open subset  of $E$ is the set of second
category.
\end{defn}
\begin{rem}
Following Solovay \cite{Solovay1970}, let '$ZF$' denote the axiomatic
set theory of Zermelo-Fraenkel and let '$ZF$ + $DC$' denote the system
obtained by adjoining a weakened form of the axiom of choice, $DC$,
(see p. 52 of \cite{Solovay1970} for a formal statement of $DC$).
The system $ZF$ + $DC$ is important because all the positive results of
elementary measure theory and most of the basic results of elementary
functional analysis, except for the Hahn-Banach theorem and other such
consequences of the axiom of choice, are provable in $ZF$ + $DC$. In
particular, the Baire category theorem for complete metric spaces and
the closed graph theorem for operators between Fr\'{e}chet spaces are
provable in $ZF$ + $DC$.
\end{rem}

Let $ZFC$ be Zermelo-Frankel set theory together with the axiom of choice.
In the presence of the axiom of choice, we identify  cardinals with  initial ordinals. A cardinal $\aleph$ is regular, if each order unbounded
subset of $\aleph$ has power $\aleph$. A cardinal $\aleph$ is inaccessible, if  it is regular,  uncountable,  and for $\aleph^{'} < \aleph$,
$2^{\aleph^{'}} < \aleph$.

Let $I$ be the statement: ''{\it There is an inaccessible cardinal}''.

\begin{lem}(\cite{Solovay1970}, THEOREM 1, p.1) Suppose that there is a transitive $\epsilon$-model of $ZFC$ + $I$. Then there is a transitive
$\epsilon$-model of $ZF$ in which the following pro positions are valid .

(1)   The principle  of dependent  choice $(DC)$.

(2)	Every set of reals is Lebesgue  measurable $(LM)$.

(3)	Every set of reals has the property of  Baire.

(4)	Every uncountable  set of reals contains a perfect  subset.

(5)	Let $\{A_x : x \in R\}$  be an indexed family of non-empty set of reals with  index set the reals.  Then there are Borel functions
$h_1,h_2$
mapping  $R$ into $R$ such that

(a)	$\{x : h_1( x ) \notin 	A_x\}$  has  Lebesgue   measure  zero.

(b)	$\{x : h_2( x ) \notin 	A_x\}$  is of first  category.

\end{lem}

Let $(X,S,\mu)$  be a probability  space  and   $H$  a separable Banach space over the  vector field $R$. Let denote by  $H(X,S,\mu)$ a class of
all  $\mu$-measurable   mappings  of the $X$  into $H$.  As usual, we  will identify equivalent  mappings, i.e., we will consider  the
factor-space with respect   to an   equivalent  relation on $H(X,S,\mu)$.  If   $\mu$  is separable and  the metric $\rho$  in  $H(X,S,\mu)$  is
defined by
$$
\rho(f,g)=\int_{X}||f-g||/(1+||f-g||)d \mu,
$$
 then  $H(X,S,\mu)$      stands  a Polish topological vector space without  isolated points.
Suppose that $\mu$  is separable and diffused.  Let $M$ be a vector subspace of  $H(X,S,\mu)$   which contains the class of all measurable step
functions.
A linear operator  $I:M \to H$  is called  a generalized  integral    for  a vector-function if for each  step function $g:X \to H$   the equality
$$
I(g)=\sum_{i=1}^na_i\mu(X_i),
$$
holds true, where $(X_i)_{1 \le i \le n}$ is a measurable partition of $X$  and  $g$ takes the value $a_i$  at points
of $X_i$  for each $i ( 1 \le i \le n)$.
It is obvious that  Reiman, Lebesgue and Danjua  integrals  are partial cases of the  generalized  integral    for vector-function and  they  are
realized when under $H$  is taken the real axis  $R$ considered as a Banach space over vector field  $R$.

Let denote by $(SM)$ the transitive $\epsilon$-model of $ZF$ which comes from  Lemma 2.4.
\begin{lem}(\cite{Pan85-3}, Theorem, p. 34 )$(SM)$ The domain of an arbitrary  generalized  integral for a vector-function is of first category in $H(X,S,\mu)$.
\end{lem}

\begin{defn}Let $\mathbf{V}$ be a Polish topological vector space. Let $K$ be the class of all probability diffused Borel
measures defined on the Borel $\sigma$-algebra $\mathcal{B}(\mathbf{V})$. We
denote by $\mathcal{{B}}(\mathbf{V})^{\mu}$  the completion of ${\mathcal{B}}(\mathbf{V})$
with respect to the measure $\mu$ for $\mu \in K$. A set $E
\subset \mathbf{V}$ is called  universally measurable if $E \in \cap_{\mu \in
K}{{\mathcal{B}}(\mathbf{V})}^{\mu}$.
\end{defn}
\begin{lem}( [9], Remark 3.4, p. 40 ) (SM) Let $\mathbf{V}$ be a complete separable metric space. Then every subset of $\mathbf{V}$ is universally measurable.

\end{lem}

\begin{defn}[ \cite{HSY92}, Definition 1, p. 221]  A measure $\mu$ is said to be transverse to a universally measurable  set $S \subset
\mathbf{V} $ if the
following two conditions hold:

(i)~There exists a compact set $U \subset \mathbf{V} $ for which
$0 < \mu(U) <1$.

(ii)~ $\overline{\mu}(S + v) = 0$ for every $v \in \mathbf{V} $, where $\overline{\mu}$ denotes a usual completion of the measure $\mu$.
\end{defn}

\begin{defn}[ \cite{HSY92}, Definition 2, p. 222]
A universally measurable  set $S \subset \mathbf{V} $ is called shy  if there exists
a measure transverse to $S$. More generally, a subset of
$\mathbf{V} $ is called shy if it is contained in a universally measurable shy set.
The complement of a shy set is called a prevalent set.
\end{defn}

\begin{rem} Note that the class of all shy sets in a olish topological vector space $\mathbf{V}$ constitutes a $\sigma$-ideal(see, \cite{HSY92}, Fact $3^{''}$,  p. 223).
\end{rem}

\section{Main result}

\begin{thm}  In the model $(SM)$  the domain of an arbitrary  generalized  integral for a vector-function in $H(X,S,\mu)$  is shy.
\end{thm}

\begin{proof} Let $I$ be a  generalized  integral for a vector-function in $H(X,S,\mu)$ and $M$ be it's domain. Since $H(X,S,\mu)$ is a Polish
topological vector space, by Lemma 2.7 we know that  $M$ is universally measurable. By Lemma 2.5 we know that $M$ is the set of the first
category in $H(X,S,\mu)$. By Remark 2.3 we know that $H(X,S,\mu)$ is the Baire space which implies that  $H(X,S,\mu) \setminus M \neq \emptyset$. Let $v \in H(X,S,\mu)
\setminus M$. Let us show that  $v$ spans a line $L$
such that every translate of $L$ meets $M$ in at most one point. Indeed,
assume the contrary. Then there will be an element $z \in H(X,S,\mu)$ and two different
parameters $t_1,t_2 \in \mathbf{R}$ such that $z+t_1 v \in M$ and $z+t_2 v \in M$.  Since $M$ is a vector subspace
of $H(X,S,\mu)$ we deduce that $(t_2-t_1)v \in M$. Using the same
argument we claim that $v \in M$ because $t_2-t_1\neq
0$, but the latter relation is the contradiction. Hence the Lebesgue measure supported on $L$ is transverse
to a universally measurable set $M$ which means that $M$ is shy.

This ends the proof of the theorem.

\end{proof}

By using Lemma  2.5  and Theorem 3.1 we get

\begin{thm}~{\it  In the model $(SM)$  the
duality principle between the measure and the Baire category
is valid with respect to the sentence $P$ defined by

'' {\it  The domain of an arbitrary  generalized  integral for a vector-function is  of  first category in $H(X,S,\mu)$}.''

}

\end{thm}

{\bf Acknowledgement}. The main result of the present paper was reported on the Section of Foundations of Mathematics and
Mathematical Logic of the
XXIX Enlarged Sessions of the Seminar of
Ilia Vekua Institute of Applied Mathematics of
Ivane Javakhisvili Tbilisi State University.

\end{document}